\documentclass{article}
\usepackage{graphics}
\usepackage{graphicx}
\usepackage{latexsym}
\usepackage{amssymb}
\usepackage{amsmath}

\newtheorem{theorem}{Theorem}[section]

\newtheorem{proposition}[theorem]{Proposition}
\newtheorem{corollary}[theorem]{Corollary}

\newtheorem{preexample}{Example}[section]

\newtheorem{preremark}{Remark}
\newenvironment{remark}{\begin{preremark}\rm}{\end{preremark}}

\newcommand{\qed }{ \hfill $\Box$ }

\begin{document}

\begin{center}
  {\Large Transport Theorem and Continuity Equations: a Stochastic Approach\\}
\end{center}

\vspace{0.3cm}

\begin{center}

{\large Pedro Catuogno \footnote{\textit{Departamento de Matem\'{a}tica, Universidade Estadual de Campinas,\\ 13.081-970 -
 Campinas - SP, Brazil. e-mail: pedrojc@ime.unicamp.br}} and Sim\~ao N. Stelmastchuk \footnote{
\textit{Departamento de  Matem\'atica, Universidade Estadual do Paran\'a,\\ 84600-000 -  Uni\~ao da Vit\' oria - PR,
Brazil. e-mail: simnaos@gmail.com}}}
\end{center}

\vspace{0.3cm}

\begin{abstract}
  We give a stochastic generalization of transport theorem on smooth manifold. Furthermore, we deduce a system of continuity equation and present  some application on torus.
\end{abstract}


\vspace{0.3cm} \noindent {\bf MSC2010 subject classification:} 76M35, 60K40

\section{Introduction}

The integration on smooth manifold is essential on the development of Mathema-tics and Science. One can integrate a form on a smooth manifold over a smooth flow. This integration yields the theory of fluid dynamics on smooth manifolds. For example, equations of motion, Euler's equation and Navier-Stokes's equation are deduced from this kind of integration( see for instance \cite{chorin} and \cite{marsden}). Moreover, in Mechanics there are many results yielded from integration on smooth manifolds(see for example \cite{arnold}).

It is possible to make an integration over a stochastic flow on smooth manifolds, see Bismut \cite{bismut} and Kunita \cite{kunita}. Recently, L\'azaro-Cam\'i and Ortega used this kind of integral in order to study the mechanic flows, see \cite{lazaro1} and \cite{lazaro2}.

The subject of this work is to give a stochastic transport theorem on smooth manifolds. In fact, we consider $X^{0}, X^{1}, \ldots, X^{m}$  time-dependent smooth vector fields, $\theta$ a time-dependent p-form  with compact support, $\sigma_p$ a $p$-simplex and $(B^1_t,\ldots,B^m_t)$ a Brownian motion in $\mathbb{R}^{m}$. Let $\phi_t$ be the flow generated by the Stratonovich stochastic differential equation,
\begin{eqnarray*}
  dx & = & X^0(t,x) dt + X^i(t,x)\circ\! dB^{i}_t \\
  x(0) & = & x.
\end{eqnarray*}
Then we have the It\^o's formula (see Theorem 4.2 \cite{kunita} or Theorem 3.7 in \cite[ch.IV]{bismut}),
\[
  \int_{\phi_t(\sigma_p)} \theta = \int_{\sigma_p} \theta + \int_{0}^{t}\left(\int_{\phi_{s}(\sigma_p)}\frac{\partial \theta}{\partial t}\right)ds + \sum_{k=0}^{m} \int_{0}^{t} \left(\int_{\phi_{s}(\sigma_p)} \!\!\!\! L_{X^k_s}\theta \right) \circ \! dB^k_s.
\]
If the smooth manifold has a volume form $\mu$, we prove a stochastic generalization of transport theorem
\[
  d\int_{\phi_{t}(\sigma_n)} f_t\mu = \int_{\phi_{t}(\sigma_n)}\left(\frac{\partial f_t}{\partial t} + \sum_{k=0}^{m} div_{\mu}(f_t X^k_t) \circ \! B^k_t \right) \mu,
\]
where $f$ is a smooth function.

In the case of a compact oriented Riemannian manifold we obtain a system of continuity equations for the mass density $\rho_{t}$,
\begin{eqnarray*}
  \frac{\partial \rho_t}{\partial t} + div_{\mu}(\rho_t X_0(t)) & = & 0 \\
  div_{\mu}(\rho_t X_k(t))& = & 0.
\end{eqnarray*}

About the structure of this paper, in section 2 we proceed with study of the stochastic flows acting on integral, more specifically, we develop the  necessary tools to show the It\^o's formula above. Then in section 4 we prove our main result, a stochastic generalization of transport theorem and  we give some applications.

\section{Stochastic flows acting on Integral}

In this section, our mean is to construct the integral over forms upon a stochastic flow. We begin by introducing some notations. Let $(\Omega, \mathcal{F},(\mathcal{F}_{t})_{t\geq0}, \mathbb{P})$ be a \linebreak probability space which satisfies the usual hypotheses(see for instance \cite[ch.I]{kunita}).

Let $M$ be a n-dimensional smooth manifold, $X^{0}, X^{1}, \ldots, X^{m}$ time-dependent smooth vector fields on $M$ and $B_t=(B^1_t, \ldots,B^m_t)$ a m-dimensional Brownian motion in $\mathbb{R}^{m}$. We consider the Stratonovich stochastic differential equation on $M$ given by
\begin{eqnarray} \label{sde}
  dx & = & X^0(t,x) dt + X^k(t,x)\circ\! dB^{k}_t \\
  x(0) & = & x, \, x \in M. \nonumber
\end{eqnarray}
It is well known that there exists an unique solution of this equation with a maximal time $T(x)$. A complete study about SDE (\ref{sde}) is founded in \cite{kunita}. We denote the solution of SDE (\ref{sde}) by $\phi_{t}(\omega,x)$ or, simply, $\phi_t$.

Let $\theta$ be a time-dependent p-form on $M$ with compact support and $\sigma_p$ a $p$-simplex in $M$, we denote by $\int_{\phi_t(\sigma_p)} \theta$ the real semimartingale $ \int_{\sigma_p} \phi_{t}^{*}\theta$.


We recall that the push-forward for a smooth vector field $X$ on $M$  by a diffeomorphism $\phi$ is given by
\[
  (\phi_* X)_y = \phi_{\phi^{-1}(y)*}X(\phi^{-1}(y)).
\]


We can now rephrase Theorem 4.3 in \cite[ch.III]{kunita} on time-dependent forms as follows.

\begin{theorem}\label{itoformulaforforms}
  Let $M$ be a n-dimensional smooth manifold and $\phi_{t}$ the flow in $M$ given by SDE (\ref{sde}). Then for a time-dependent p-form $\theta$ with compact support we have
  \begin{eqnarray*}
    \phi_{t}^*\theta - \theta
    & = & \int_{0}^{t}\phi_{s}^*\frac{\partial \theta}{\partial t}ds + \sum_{k=0}^{m}\int_{0}^{t}\left(\phi_{s}^{*}L_{X^k_s}\theta\right) \circ\! d B^{k}_{s}.\\
  \end{eqnarray*}
\end{theorem}

In order to write It\^o's formulas to real semimartingales $\int_{\phi_t(\sigma_p)} \theta$, our next step is to show a Fubini's Theorem.

\begin{proposition}\label{fubinimanifold}
  Let $M$ be n-dimensional smooth manifold, $\theta$ a time-dependent p-form on $M$ with compact support and $B_t$ a real Brownian motion. Then
  \[
    \int_{\sigma_p} \left(\int_{0}^{t} \phi_{s}^* \theta \circ\! d B_s \right)  = \int_{0}^{t} \left(\int_{\sigma_p}  \phi_{s}^* \theta\right) \circ\! d B_s.
  \]
\end{proposition}
\begin{proof}
  Let $\Delta=\{0=t_{0}< \ldots< t_{n}=T\}$ be a partition of the interval $[0,T]$, $t^*= \frac{t_{j+1}+t_{j}}{2}$ and $\Delta B(\omega) = (B_{t_{j+1}}(\omega) - B_{t_{j}}(\omega))$. It is well known that, in $L^2([0,T]\times \Omega)$,
  \[
    \lim_{j \to \infty} \mathbb{E}\left(\left|\int_{0}^{t} \left(\int_{\sigma_p}  \phi_{s}^* \theta\right) \circ\! d B_{s} - \sum_{j}\left(\int_{\sigma_p}  \phi_{t*}^* \theta\right)\Delta B(\omega)\right|^2\right)= 0.
  \]
  From uniqueness of limit we get
  \[
    \int_{\sigma_p} \left(\int_{0}^{t} \phi_{s}^* \theta \circ\! d B_s \right) = \int_{0}^{t} \left(\int_{\sigma_p}  \phi_{s}^* \theta\right) \circ\! d B_{s}.
  \]
  \qed
\end{proof}

\begin{corollary}\label{theoremofintegralequation}
  Let $M$ be a n-dimensional smooth manifold and $\phi_{t}$ the flow in $M$ given by SDE (\ref{sde}). Then for a time-dependent p-form $\theta$ with compact support we have
  \begin{enumerate}
    \item
      \begin{equation}\label{theitointegraleq2}
        \int_{\phi_t(\sigma_p)} \theta = \int_{\sigma_p} \theta + \int_{0}^{t}\left(\int_{\phi_{s}(\sigma_p)}\frac{\partial \theta}{\partial t}\right)ds + \sum_{k=0}^{m} \int_{0}^{t} \left(\int_{\phi_{s}(\sigma_p)} \!\!\!\! L_{X^k_s}\theta \right) \circ\! d B^k_s.
      \end{equation}
    \item
      \begin{eqnarray}\label{theitointegraleq1}
        \int_{\phi_t(\sigma_p)} \theta
        & = & \int_{\sigma_p} \theta +  \int_{0}^{t} \left(\int_{\phi_{s}(\sigma_p)}(\frac{\partial }{\partial t} + \dfrac{1}{2} \sum_{k=1}^{m}L^2_{X^k_s}+ L_{X^0_s})\theta \right) ds \nonumber \\
        & + & \sum_{k=1}^{m}\int_{0}^{t} \left(\int_{\phi_{s}(\sigma_p)} \!\!\!\! L_{X^k_s}\theta \right) dB^k_s.
      \end{eqnarray}
  \end{enumerate}
\end{corollary}
\begin{proof}
  {\bf 1.}
  From Theorem \ref{itoformulaforforms} we obtain
  \[
    \int_{\phi_t(\sigma_p)} \theta  = \int_{\sigma_p} \theta +  \int_{\sigma_p}\int_{0}^{t}\phi_{s}^*\frac{\partial \theta}{\partial t}ds + \int_{\sigma_p} \sum_{k=0}^{m} \int_{0}^{t} \phi_{s}^{*}L_{X^k_s}\theta \circ\! d B^k_s.
  \]
  Applying Proposition \ref{fubinimanifold} in the right side we see that
  \begin{equation}\label{theoremofintegralequationeq1}
    \int_{\sigma_p} \sum_{k=0}^{m} \int_{0}^{t} \phi_{s}^{*}L_{X^k_s}\theta \circ\! d B^k_s
    = \sum_{k=0}^{m}\int_{0}^{t} \left(\int_{\sigma_p} \phi_{s}^{*}L_{X^k_s}\theta \right)\circ\! d B^k_s.
  \end{equation}
  Analogously, we can see that
  \begin{eqnarray}\label{theoremofintegralequationeq2}
    \int_{\sigma_p}\int_{0}^{t}\phi_{s}^*\frac{\partial \theta}{\partial t}ds = \int_{0}^{t}\int_{\phi_{s}(\sigma_p)}\frac{\partial \theta}{\partial t}ds.
  \end{eqnarray}
  Therefore the formula (\ref{theitointegraleq2}) follows from (\ref{theoremofintegralequationeq1}) and (\ref{theoremofintegralequationeq2}). \\
  {\bf 2.}
  We first apply the Stratonovich-It\^o conversion formula in stochastic integrals of the formula (\ref{theitointegraleq2}). It follows that
  \[
    \int_{0}^{t}\!\!\! \left(\int_{\phi_{s}(\sigma_p)}\!\!\!\!\!\!\! L_{X^k_s}\theta \right) \circ\! d B^k_s\! =\!\! \int_{0}^{t}\!\!\! \left(\int_{\phi_{s}(\sigma_p)}\!\!\!\!\!\! L_{X^k_s}\theta \right) dB^k_s + \dfrac{1}{2}[\left(\int_{\phi_{s}(\sigma_p)}\!\!\!\!\!\! L_{X^k_s}\theta\right), B^k_s].
  \]
  Using the formula (\ref{theitointegraleq2}) for $\int_{\phi_{s}(\sigma_p)} L_{X^k_s}\theta $, $k=1, \ldots,n$, we obtain
  \begin{eqnarray*}
    \dfrac{1}{2}[\left(\int_{\phi_{s}(\sigma_p)}L_{X^k_s}\theta\right), B^k_s]
    & = & \dfrac{1}{2} [\sum_{l=0}^{m} \int_{0}^{t} \left(\int_{\phi_{s}(\sigma_p)} L_{X^l_s}L_{X^k_s}\theta \right) \circ\! d B^l_s, B^k_s]\\
    & = & \dfrac{1}{2} \int_{0}^{t} \left(\int_{\phi_{s}(\sigma_p)} L_{X^k_s}L_{X^k_s}\theta \right) ds.
  \end{eqnarray*}
  Therefore we conclude that
  \[
    \int_{0}^{t} \left(\int_{\phi_{s}(\sigma_p)}\!\!\!\!\!\! L_{X^k_s}\theta \right) \circ\! d B^k_s
    = \int_{0}^{t} \left(\int_{\phi_{s}(\sigma_p)} \!\!\!\! L_{X^k_s}\theta \right) dB^k_s
    + \dfrac{1}{2} \int_{0}^{t} \left(\int_{\phi_{s}(\sigma_p)}\!\!\!\!\! L^2_{X^k_s}\theta \right) ds.
  \]
  Thus substituting the equality above in the formula (\ref{theitointegraleq2}) yields the formula (\ref{theitointegraleq1}).\qed
\end{proof}

This Corollary is reformulation of Theorem 3.7 in \cite[ch.IV]{bismut} in terms of time-dependent forms.

A direct consequence of the formula (\ref{theitointegraleq1}) is that if a time-dependent p-form $\theta$ with compact support satisfies the differential equation
\[
  \left\{
  \begin{array}{ccl}
    \dfrac{\partial\theta} {\partial t} & = & - \left(\dfrac{1}{2} \sum_{k=1}^{m}L^2_{X^k_s}+ L_{X^0_s}\right)\theta,\\
    \theta(0,x) & = & \theta_0(x)
  \end{array}
  \right.
\]
then
\[
  \int_{\phi_t(\sigma_p)} \theta = \int_{\sigma_p} \theta_0 +  \sum_{k=1}^{m}\int_{0}^{t} \left(\int_{\phi_{s}(\sigma_p)} \!\!\!\! L_{X^k_s}\theta \right) dB^k_s
\]
is a real martingale. Thus taking expectation we conclude that
\[
  \int_{\sigma_p} \theta_0 = \mathbb{E}\left(\int_{\phi_t(\sigma_p)} \theta\right).
\]

\section{Transport Theorem and Continuity Equations}

In this section we develop some results in stochastic fluid mechanics. We follow close the presentation of Abraham, Ratiu and Marsden \cite{marsden}. Our first result is a stochastic generalization of transport theorem.
\begin{theorem}\label{transporttheorem}
  Let $M$ be a smooth manifold, $\mu$ a volume form and $\phi_{t}$ the flow generated by SDE (\ref{sde}). For a smooth function $f: [0,\infty)\times M \rightarrow \mathbb{R}$, we have
  \[
    \int_{\phi_{t}(\sigma_n)}\!\!\!\! f_t\mu = \int_{\sigma_n}\!\!\! f_0 \mu + \int_{0}^{t}\!\!\!\left(\int_{\phi_{s}(\sigma_n)}\frac{\partial f_s}{\partial t} \mu \right)ds + \sum_{k=0}^{m} \int_{0}^{t} \left(\int_{\phi_{s}(\sigma_n)}  div_{\mu}(f_s X^k_s)\mu \right) \circ\! d B^k_s
   \]
  for any $n$-simplex $\sigma_n$.
\end{theorem}
\begin{remark}
  Being $\mu$ a volume form on $M$, we can generated a measure $m_\mu$ associated to it. Therefore the functions $f$ adopted in Theorem above are smooth functions in $L^1(M,\mu_M)$, as required in \cite{marsden}.
\end{remark}
\begin{proof}
  We first apply the Corollary \ref{theoremofintegralequation} to obtain
  \[
    \int_{\phi_t(\sigma_n)}\!\!\!\!\!\!\! f_t \mu\! = \! \int_{\sigma_n}\!\!\! f_0 \mu + \int_{0}^{t}\!\!\!\int_{\phi_{s}(\sigma_n)}\frac{\partial}{\partial t}(f_s \mu)ds + \sum_{k=0}^{m} \int_{0}^{t}\!\! \left(\int_{\phi_{s}(\sigma_n)}\!\!\!\!\!\!\!\!\!  L_{X^k_s}(f_s \mu) \right) \circ\! d B^k_s.
  \]
  As  $L_{X^k_t}(f_t \mu) = div_{\mu}(f_t X^k_t)\mu$ we have
  \[
    \int_{\phi_t(\sigma_n)}\!\!\!\!\!\!\! f_t \mu = \int_{\sigma_n}\!\!\! f_0 \mu + \int_{0}^{t}\!\!\!\int_{\phi_{s}(\sigma_n)}\frac{\partial f_t}{\partial t} \mu ds + \sum_{k=0}^{m} \int_{0}^{t} \left(\int_{\phi_{s}(\sigma_n)}  div_{\mu}(f_s X^k_s)\mu \right) \circ\! d B^k_s,
  \]
  which proves the theorem.\qed
\end{proof}

\begin{theorem}\label{transporttheorem2}
  Let $M$ be a smooth manifold, $\mu$ a volume form and $\phi_{t}$ the flow generated by SDE (\ref{sde}). For a smooth function $f: [0,\infty)\times M \rightarrow \mathbb{R}$, we have
  \begin{eqnarray*}
    d \mathbb{E}\left(\int_{\phi_t(\sigma_n)} f_t \mu\right)\!\!\!
    & = & \mathbb{E}\left(\int_{\phi_{s}(\sigma_n)}\left(\frac{\partial f_s }{\partial t} + f_s(div_\mu(X^0_s) +\dfrac{1}{2} \sum_{k=1}^{m} div_\mu(X^k_s)^2)\right)\mu\right)\\
    & + & \dfrac{1}{2} \sum_{k=1}^{m}\mathbb{E}\left(\int_{\phi_{s}(\sigma_n)}\left(2X^k_s(f_s)div_\mu(X^k_s) + L_{X^k_s}(f_s div_\mu(X_s^k))\right)\mu\right)\\
    & + & \mathbb{E}\left(\int_{\phi_{s}(\sigma_n)}\mathcal{L}_s(f_s)\mu\right).
  \end{eqnarray*}
  for any $n$-simplex $\sigma_n$, where $\mathcal{L}_s = X_s^0 + \dfrac{1}{2} \sum_{k=1}^{m}X^k_s$ is the infinitesimal generator of sde (\ref{sde}).
\end{theorem}
\begin{proof}
  We first apply the Corollary \ref{theoremofintegralequation} to obtain
  \begin{eqnarray*}
    \int_{\phi_t(\sigma_n)} f_t \mu
    &= & \int_{\sigma_n} f_0 \mu +  \int_{0}^{t} \left(\int_{\phi_{s}(\sigma_n)}(\frac{\partial }{\partial t} + \dfrac{1}{2} \sum_{k=1}^{m}L^2_{X^k_s}+ L_{X^0_s})(f_s \mu) \right) ds  \\
    & + & \sum_{k=1}^{m}\int_{0}^{t} \left(\int_{\phi_{s}(\sigma_n)} \!\!\!\! L_{X^k_s}(f_s \mu) \right) dB^k_s.
  \end{eqnarray*}
  As  $L_{X^k_t}(f_t \mu) = div_{\mu}(f_t X^k_t)\mu$ we have
  \begin{eqnarray*}
    \int_{\phi_t(\sigma_n)} f_t \mu
    &= & \int_{\sigma_n} f_0 \mu +  \int_{0}^{t} \left(\int_{\phi_{s}(\sigma_n)}(\frac{\partial f_s}{\partial t} + div_{\mu}(f_s X_s^0) + \dfrac{1}{2} \sum_{k=1}^{m}div_\mu(div_\mu(f_sX_s^k)X_s^k))\mu\right) ds  \\
    & + & \sum_{k=1}^{m}\int_{0}^{t} \left(\int_{\phi_{s}(\sigma_n)} \!\!\!\! div_{\mu}(f_t X^k_t)\mu \right) dB^k_s.
  \end{eqnarray*}
  Taking the expectation we have
  \begin{eqnarray*}
    \mathbb{E}\left(\int_{\phi_t(\sigma_n)} f_t \mu\right)\!\!\!
    & = & \int_{\sigma_n} f_0 \mu \\
    & + & \!\!\!\!\!\!\! \int_{0}^{t} \mathbb{E}\left(\int_{\phi_{s}(\sigma_n)}\!\!\!\left(\frac{\partial f_s }{\partial t} + div_{\mu}(f_s X_s^0) + \dfrac{1}{2} \sum_{k=1}^{m}div_\mu(div_\mu(f_sX_s^k)X_s^k)\right)\mu\right) ds.
  \end{eqnarray*}
  Differentiating we obtain
  \begin{eqnarray*}
    d \mathbb{E}\left(\int_{\phi_t(\sigma_n)} f_t \mu\right)\!\!\!
    & = &\!\!\!\!\! \mathbb{E}\left(\int_{\phi_{s}(\sigma_n)}\left(\frac{\partial f_s }{\partial t} + div_{\mu}(f_s X_s^0) + \dfrac{1}{2} \sum_{k=1}^{m}div_\mu(div_\mu(f_sX_s^k)X_s^k)\right)\mu\right).
  \end{eqnarray*}
  Now applying the properties of $div_\mu$ we can deduce that
  \begin{eqnarray*}
    d \mathbb{E}\left(\int_{\phi_t(\sigma_n)} f_t \mu\right)\!\!\!
    & = & \mathbb{E}\left(\int_{\phi_{s}(\sigma_n)}\left(\frac{\partial f_s }{\partial t} + f_s(div_\mu(X^0_s) +\dfrac{1}{2} \sum_{k=1}^{m} div_\mu(X^k_s)^2)\right)\mu\right)\\
    & + & \dfrac{1}{2} \sum_{k=1}^{m}\mathbb{E}\left(\int_{\phi_{s}(\sigma_n)}\left(2X^k_s(f_s)div_\mu(X^k_s) + L_{X^k_s}(f_s div_\mu(X_s^k))\right)\mu\right)\\
    & + & \mathbb{E}\left(\int_{\phi_{s}(\sigma_n)}(X_s^0(f_s) + \dfrac{1}{2} \sum_{k=1}^{m}X^k_s(f_s)) \mu\right).
  \end{eqnarray*}
  \qed
\end{proof}
\begin{corollary}
  Under hypothesis of Theorem \ref{transporttheorem2}, furthermore, if the vector fields $X_0, \ldots, X_m$ in sde (\ref{sde}) are divergence free, then
  \begin{eqnarray*}
    d \mathbb{E}\left(\int_{\phi_t(\sigma_n)} f_t \mu\right)\!\!\!
    & = & \mathbb{E}\left(\int_{\phi_{s}(\sigma_n)}\left(\frac{\partial f_s }{\partial t} + \mathcal{L}_s(f_s)\right)\mu\right).
  \end{eqnarray*}
  for any $n$-simplex $\sigma_n$, where $\mathcal{L}_s = X_s^0 + \dfrac{1}{2} \sum_{k=1}^{m}X^k_s$ is the infinitesimal generator of sde (\ref{sde}).
\end{corollary}

We use Theorem \ref{transporttheorem} to give a stochastic version of continuity equation. Let $M$ be a compact oriented Riemannian manifold and $\mu$ the Riemannian volume form. We observe that the mechanical interpretation of the right side of SDE (\ref{sde})
\[
  dx  =  X_0(t,x) dt + X^k(t,x)\circ\! d B^{k}_t
\]
is the velocity field of the fluid.

For each time $t$, we shall assume that the fluid has a well-defined mass density $\rho_t(x) = \rho(t,x)$. Taking any open set $U$ in $M$ we assume that the mass of fluid in $U$ at time $t$ is given by
\[
  m(t,U) = \int_{U} \rho_{t} \mu.
\]

Assuming that {\it mass is neither created nor destroyed}. The meaning of this assumption to the open set $U$ is
\[
  \int_{\phi_t(U)} \rho_{t} \mu = \int_{U} \rho_{0} \mu.
\]
Applying Theorem \ref{transporttheorem} we obtain
\[
  \int_{0}^{t}\!\!\!\left(\int_{\phi_{s}(U)}\frac{\partial \rho_s}{\partial t} \mu \right)ds + \sum_{k=0}^{m} \int_{0}^{t} \left(\int_{\phi_{s}(U)}  div_{\mu}(\rho_s X^k_s)\mu \right) \circ\! d B^k_s = 0
\]
We now apply the quadratic variation with respect to $B_t^l, l=1, \ldots, n,$ to get
\[
  \left[ \int_{0}^{t}\!\!\!\left(\int_{\phi_{s}(U)}\frac{\partial \rho_s}{\partial t} \mu \right)ds + \sum_{k=0}^{m} \int_{0}^{t} \left(\int_{\phi_{s}(U)}  div_{\mu}(\rho_s X^k_s)\mu \right) \circ\! d B^k_s, B^l_t\right]  =  0.
\]
Thus we conclude, for k=1,\ldots m, that
\[
  \int_{\phi_{s}(U)}  div_{\mu}(\rho_s X^k_s)\mu = 0.
\]
Being $U$ an arbitrary open, the continuity equations are given by
\begin{eqnarray}\label{continuityequation2}
  \frac{\partial \rho_t}{\partial t} + div_{\mu}(\rho_t X^0_t) & = & 0 \nonumber\\
  div_{\mu}(\rho_t X^k_t)& = & 0.
\end{eqnarray}


To end this work we apply the continuity equations on torus. The stochastic flow on torus has been study in many works, we cite for example \cite{cruzeiro1} and \cite{cruzeiro2}. In our case  we use a stochastic flow with divergence free vector fields.

Let $k=(k_1,k_2) \in \mathbb{Z}^2$ and $\theta = (\theta_1, \theta_2) \in \mathbb{T}$. Define the following vector fields on torus $\mathbb{T}$
\begin{eqnarray}\label{vectorAandB}
  A_k & = & k_2 \cos (k \cdot \theta)\partial_1 - k_1\cos(k \cdot \theta) \partial_2 \nonumber\\
  B_k & = & k_2 \sin (k \cdot \theta)\partial_1 - k_1\sin(k \cdot \theta) \partial_2,
\end{eqnarray}
where $k \cdot \theta = k_1 \theta_1 + k_2 \theta_2$ and $\partial_1, \partial_2$ are the coordinate vector field on $\mathbb{T}$. In \cite{cruzeiro1} is showed that $A_k$ and $B_k$ are divergence free.

Let $(B^1_t,B^2_t)$ be a $\mathbb{R}^2$-valued standard Brownian motion in $\mathbb{R}^2$. Now, consider a divergence free vector field $u(t) \in \mathbb{T}$ and the following Stratonovich stochastic differential equation
\[
  dg^k(t) = u_tdt + A_k dB^1_t + B_k dB^2_t.
\]
We want to study the continuity equations (\ref{continuityequation2}) for the flow above and a volume form $\mu$ on torus. In fact, the continuity equations are given by
\begin{eqnarray*}
  \frac{\partial \rho^k_t}{\partial t} + div_{\mu}(\rho^k_t u_t) & = & 0 \\
  div_{\mu}(\rho^k_t A_k)& = & 0\\
  div_{\mu}(\rho^k_t B_k)& = & 0.
\end{eqnarray*}
From the divergence free property we obtain
\begin{equation}\label{flowtorustime}
  \rho^k_t= \rho^k_0 + \int_{0}^{t} u_s(\rho^k_s) ds
\end{equation}
and
\[
  <\nabla\rho^k_t,A_k> =  0 \ \ \textrm{and} \ \ <\nabla\rho^k_t,B_k> =  0.
\]
Applying (\ref{vectorAandB}) in the equalities above we obtain
\begin{eqnarray*}
  \cos (k \cdot \theta)<(k_2\partial_1,-k_1\partial_2),\nabla\rho^k_t> & = & 0\\
  \sin (k \cdot \theta)<(k_2\partial_1,-k_1\partial_2),\nabla\rho^k_t> & = & 0
\end{eqnarray*}
and, consequently, $<(k_2\partial_1,-k_1\partial_2),\nabla\rho^k_t>  = 0$.

Being $\theta$ arbitrary, we conclude that $\nabla\rho^k_t= (0,0)$. It implies that $\partial_1\rho^k_{t} = 0$ and $\partial_2\rho^k_{t} = 0$. Since torus is a connected manifold, $\rho^k_{t}(\theta_1,\theta_2) = \rho^k_{t}(0,0)$. Therefore applying $u_t$ in $\rho^k_{t}$ give us that $u_t(\rho^k_{t}) = 0$. We conclude from the equation (\ref{flowtorustime}) that
  \[
    \rho^k_{t}(\theta_1,\theta_2) = \rho^k_{0}(0,0).
  \]


\end{document}